\documentclass{amsart}
\usepackage{amsmath, amssymb, amsthm}
\usepackage{amscd}
\usepackage[dvips]{graphicx}
\usepackage{enumerate}

\title{A generalization of Dirichlet's unit theorem}
  \author[Fili]{Paul Fili}
 \address{Department of Mathematics\\ University of Rochester, Rochester, NY 14627}
 \email{fili@math.rochester.edu}
  \author[Miner]{Zachary Miner}
 \address{Department of Mathematics\\ University of Texas at Austin, TX 78712}
 \email{zminer@math.utexas.edu}
 \subjclass[2010]{11R04, 11R27, 46E30}
 \keywords{Dirichlet's unit theorem, algebraic numbers, Weil height.}
\date{\today}

\usepackage[pdftitle={A generalization of Dirichlet's unit theorem},pdfauthor={Fili and Miner},pdfstartview={}]{hyperref}

\usepackage[all]{xy}

\newtheorem{thm}{Theorem}

\newtheorem{prop}[thm]{Proposition}
\newtheorem{lemma}[thm]{Lemma}

\newtheorem*{thm*}{Theorem}
\newtheorem*{alg*}{Algorithm}
\newtheorem*{lemma*}{Lemma}

\theoremstyle{remark}
\newtheorem{rmk}[thm]{Remark}
\newtheorem*{rmk*}{Remark}

\newtheorem*{notation*}{Notation}
\newtheorem{example}[thm]{Example}

\theoremstyle{definition}
\newtheorem{defn}[thm]{Definition}

\numberwithin{thm}{section}

\newcommand{\mybf}{\mathbb}

\newcommand{\bR}{\mybf{R}}

\newcommand{\bN}{\mybf{N}}
\newcommand{\bZ}{\mybf{Z}}
\newcommand{\bQ}{\mybf{Q}}

\newcommand{\cP}{\mathcal{P}}

\newcommand{\cF}{\mathcal{F}}
\newcommand{\cO}{\mathcal{O}}

\newcommand{\cK}{\mathcal{K}}

\newcommand{\al}{\alpha}

\providecommand{\abs}[1]{\lvert#1\rvert}

\newcommand{\Gal}{\operatorname{Gal}}

\newcommand{\supp}{\operatorname{supp}}

\newcommand{\ra}{\rightarrow}

\newcommand{\ep}{\epsilon}

\newcommand{\Tor}{\operatorname{Tor}}
\newcommand{\Fix}{\operatorname{Fix}}
\newcommand{\Stab}{\operatorname{Stab}}
\newcommand{\ord}{\operatorname{ord}}

\newcommand{\ilim}{\mathop{\varprojlim}\limits}

\newcommand{\Qbar}{\overline{\mybf{Q}}}

\def\talltareesidedbox#1{\setbox0=\hbox{$#1$}\dimen0=\wd0 \advance\dimen0 by3pt\rlap{\hbox{\vrule height10pt width.4pt
 depth2pt \kern-.4pt\vrule height10.4pt width\dimen0 depth-10pt\kern-.4pt \vrule height10pt width.4pt depth2pt}}
 \relax \hbox to\dimen0{\hss$#1$\hss}}
\def\tareesidedbox#1{\setbox0=\hbox{$#1$}\dimen0=\wd0 \advance\dimen0 by3pt\rlap{\hbox{\vrule height8pt width.4pt
 depth2pt \kern-.4pt\vrule height8.4pt width\dimen0 depth-8pt\kern-.4pt \vrule height8pt width.4pt depth2pt}}
\relax \hbox to\dimen0{\hss$#1$\hss}}
\def\shorttareesidedbox#1{\setbox0=\hbox{$#1$}\dimen0=\wd0 \advance\dimen0 by3pt\rlap{\hbox{\vrule height7pt width.4pt
 depth2pt \kern-.4pt\vrule height7.4pt width\dimen0 depth-7pt\kern-.4pt \vrule height7pt width.4pt depth2pt}}
 \relax \hbox to\dimen0{\hss$#1$\hss}}

\newcommand{\fal}{f_\alpha}

\newcommand{\alg}{\mathrm{alg}}

\begin{document}

\begin{abstract}
 We generalize Dirichlet's $S$-unit theorem from the usual group of $S$-units of a number field $K$ to the infinite rank group of all algebraic numbers having nontrivial valuations only on places lying over $S$. Specifically, we demonstrate that the group of algebraic $S$-units modulo torsion is a $\bQ$-vector space which, when normed by the Weil height, spans a hyperplane determined by the product formula, and that the elements of this vector space which are linearly independent over $\bQ$ retain their linear independence over $\bR$.
\end{abstract}

\maketitle

\section{Introduction}
\subsection{Background}
Let $K$ be a number field with set of places $M_K$. For each $v\in M_K$ lying over a rational prime $p$ let $\|\cdot\|_v$ denote the absolute value extending the usual absolute value $\abs{\ \cdot\ }_p$ on $\bQ$. For a finite set $S\subset M_K$ containing all of the archimedean places, let $U_{K,S}\subset K^\times$ denote the usual \emph{$S$-unit group} of $K$ given by
\[
 U_{K,S} = \{ \al\in K^\times : \|\al\|_v = 1\text{ for all }v\not\in S\}.
\]
A fundamental result in algebraic number theory is Dirichlet's $S$-unit theorem, a result originally proven by Dirichlet for the units of a number field and then extended to $S$-units by Hasse and later Chevalley (see \cite[Theorem III.3.5]{Nark}):
\begin{thm*}[$S$-unit theorem]
 With the notation above, the $S$-unit group $U_{K,S}$ is a finitely generated abelian group of rank $\#S-1$. Further, under the logarithmic embedding
 \begin{equation}\label{eqn:logarithmic-embedding}
 \begin{aligned}
 l : U_{K,S}/\Tor(U_{K,S}) &\ra \bR^S\\
 \xi &\mapsto \left( \frac{[K_v:\bQ_v]}{[K:\bQ]} \log \|\xi\|_v\right)_v,
 \end{aligned}
 \end{equation}
 the image of $l$ is a lattice which spans the hyperplane
 \[
 \{x\in \bR^S : \sum_v x_v =0\}\subset \bR^S.
 \]
\end{thm*}
\noindent In other words, Dirichlet's theorem tells us that:
\begin{enumerate}[(A)]
\item \label{enum:first-conclusion} Since all algebraic numbers \emph{a priori} satisfy the product formula, and hence the condition $\sum_v x_v = 0$ under the logarithmic embedding, the $S$-unit group of a number field $K$ spans the largest possible vector space it can.
\item \label{enum:second-conclusion} The algebraic rank of the $S$-unit group is equal to the real dimension of the vector space it spans. In particular, it is not larger, as in the example of $\bZ e\oplus \bZ \pi\subset \bR$, which has algebraic rank equal to $2$ because $e$ and $\pi$ are $\bQ$-linearly independent, but spans a space of real dimension 1. In other words, under the logarithmic embedding, any vectors which are $\bR$-linearly dependent must also be $\bQ$-linearly dependent; by Northcott's theorem we easily see that we can choose a $\bQ$-vector space basis which must also generate $U_{K,S}/\Tor(U_{K,S})$.
\end{enumerate}

This paper provides a generalization of Dirichlet's theorem to the infinite dimensional space of the group of all algebraic $S$-units for a given set of places $S$. Specifically, suppose $K$ is again a number field and $S$ a finite set of places containing all of the infinite places. We demonstrate that the group
\begin{multline}\label{eqn:UalgS}
 U^\alg_S = \{\al\in \Qbar^\times : \|\al\|_w=1\text{ for all }w\in M_{K(\al)}\\
 \text{ such that }\ w\nmid v\text{ for all }v\in S\}
\end{multline}
is, modulo its torsion subgroup, a normed vector space (written multiplicatively) which satisfies conclusions \eqref{enum:first-conclusion} and \eqref{enum:second-conclusion} above, namely, that it spans the largest possible hyperplane and that its vectors which are independent over the scalar field $\bQ$ retain their linear independence over $\bR$. Naturally, unlike the $S$-unit group of a number field, this algebraic $S$-unit group $U^\alg_S$ modulo torsion is of infinite rank and thus gives rise to an infinite dimensional vector space. Thus our choice of vector space norm is essential, which of course is not the case in finite dimension, where all norms are equivalent. The natural choice for a norm is the absolute logarithmic Weil height, and we will work in the completion with respect to the Weil height introduced by Allcock and Vaaler \cite{AV}. Finally, we will show that our result (like the classical Dirichlet theorem) is sharp in the sense that if one drops the assumption that $S$ contain all of the archimedean places, then the vector space spanned may be a proper subspace of the hyperplane determined by the product formula.

Our theorem statement is in fact strictly stronger than that of the classical $S$-unit theorem, as we will demonstrate below (see Remark \ref{rmk:implies-classical} below) that the classical $S$-unit theorem is an easy consequence of ``projecting down to $K$'' from the space $U^\alg_S$. However, our proof relies on the classical $S$-unit theorem and thus does not provide an independent proof of Dirichlet's result. Nevertheless, the proof involves several novel elements, including the use of results from functional analysis (for example, Lemma \ref{lemma:2}, for which we refer the reader to \cite{AV}, is proven using the Stone-Weierstrass theorem) and the use of projection operators, particularly the $P_S$ operator defined below in \eqref{eqn:PS-defn}, to approximate numbers with respect to the height. Projection operators associated to number fields also play a key role in \cite{FM-extremal} and \cite{FM-mahler}; however, while the operators used in \cite{FM-extremal, FM-mahler} are typically continuous extensions of well-defined maps $\Qbar^\times/\Tor(\Qbar^\times) \ra \Qbar^\times/\Tor(\Qbar^\times)$ (for example, for a number field $K/\bQ$, the $P_K$ operator is a scaled analogue of the algebraic norm down to $K$), the operator $P_S$ used in the proof here differs in that it is essentially an analytic object which is defined in the completion and does \emph{not} arise from any such map of algebraic numbers modulo torsion. But even though $P_S$ will not in general fix the whole space $\Qbar^\times/\Tor(\Qbar^\times)$, it will fix the vector space of $S$-units, and its continuity as a linear operator will be seen to be essential to establishing our main results.

\subsection{Function space formulation}
We now briefly recall the constructions of Allcock and Vaaler \cite{AV} which allow us to view the the group of algebraic numbers modulo torsion as a vector space normed by the absolute logarithmic Weil height. We refer the reader to \cite{AV, FM-mahler} for more details on the constructions used here. We will let $G=\Gal(\Qbar/\bQ)$ denote the absolute Galois group throughout. Let $\cK$ denote the collection of all number fields. Observe that $\cK$ is naturally a partially ordered set under containment. For each $K\in\cK$, we endow $M_K$ with the discrete topology. Then the collection of sets
$
 \{M_K : K\in\cK\}
$
naturally forms a projective system under inclusion with the projection maps for a pair of number fields $L/K$ given by mapping the place $w$ of $L$ to the place $v$ of $K$ which one obtains by restricting $\|\cdot\|_w$ to $K$. We form the topological space
\[
 Y = \ilim_{K\in\cK} M_K,
\]
which is a locally compact, totally disconnected Hausdorff space. Notice that the points of $Y$ exactly correspond to the places of $\Qbar$. The sets $Y(K,v)$ for any $K\in\cK$ and $v\in M_K$ given by
\begin{equation}
 Y(K,v) = \{y\in Y : y|v\}
\end{equation}
naturally form a basis for the topology of $Y$. The following result, the proof of which we will defer to \S \ref{section:2}, characterizes the sets of places which most naturally generalize the sets that occur in Dirichlet's theorem:
\begin{lemma}\label{lemma:S-iff-compact-open}
 A subset $S\subset Y$ of places of $\Qbar$ is compact open if and only if there exists a number field $K$ and a finite set of places $T\subset M_K$ such that
 \[
  S = \bigcup_{v\in T} Y(K,v).
 \]
\end{lemma}

We associate to each equivalence class of a nonzero algebraic number $\al$ modulo torsion (that is, the roots of unity) a function $\fal$ given by
\begin{equation}
\begin{aligned}
 \fal : Y &\ra \bR\\
  y &\mapsto \log \|\al\|_y.
\end{aligned}
\end{equation}
As each $\al$ has a finite number of places where it has a nontrivial valuation, the functions $\fal$ are compactly supported. They are also locally constant on an appropriate collection of sets of the form $Y(K,v)$ (where, say, $\al\in K^\times$), and therefore they are continuous on $Y$. Recall that, as a divisible torsion-free abelian group, the group $\Qbar^\times/\Tor(\Qbar^\times)$ is naturally a vector space, written multiplicatively, over $\bQ$. The map
\begin{equation}\label{eqn:phi-function}
\begin{aligned}
 \phi: \Qbar^\times/\Tor(\Qbar^\times) &\hookrightarrow C_c(Y)\\
  \al &\mapsto \fal
\end{aligned}
\end{equation}
sending each equivalence class modulo torsion to the associated function is then a vector space isomorphism onto its image. We denote the image of $\phi$ by $\cF$, that is,
\begin{equation}
 \cF = \{\fal\in C_c(Y) : \al\in \Qbar^\times/\Tor(\Qbar^\times)\}.
\end{equation}
We will drop the subscript $\al$ on elements $f\in\cF$ when convenient; the reader should bear in mind that each $f\in \cF$ corresponds uniquely to some equivalence class $\al\in \Qbar^\times/\Tor(\Qbar^\times)$. Lastly, to each number field $K\in\cK$ we associate the $\bQ$-vector space which $K^\times/\Tor(K^\times)$ spans:
\begin{equation}
 V_K = \{\fal : \al^n \in K^\times/\Tor(K^\times)\text{ for some }n\in\bN\}\subset \cF.
\end{equation}
Each $V_K$ has a canonical projection $P_K:\cF\ra V_K$ associated to it which is defined and studied in \cite[\S 2.3]{FM-mahler}, to which we refer the reader. Continuous projection maps will play an important role in the proofs of our main theorems.

Allcock and Vaaler construct a Borel measure $\lambda$ on $Y$ such that for each number field $K$ and place $v$ of $K$,
\begin{equation}
 \lambda(Y(K,v)) = \frac{[K_v:\bQ_v]}{[K:\bQ]}.
\end{equation}
The $L^1$ norm of such a function with respect to this measure is then precisely twice its absolute logarithmic Weil height:
\begin{equation}
 2h(\al) = \sum_{v\in M_K} \frac{[K_v:\bQ_v]}{[K:\bQ]} \abs{\log \|\al\|_v} = \int_Y \abs{f(y)}\,d\lambda(y) = \|\fal\|_1
\end{equation}
(where we assume without loss of generality $\al\in K^\times$ for some number field $K$). The product formula now takes the form
\begin{equation}
 \int_Y \fal(y)\,d\lambda(y)=0,
\end{equation}
and thus the space $\cF$ naturally sits inside the vector space
\[
 \{f\in C_c(Y) : \int f \,d\lambda = 0\}
\]
which, with appropriate restrictions on the support of $f$, will form the natural generalization of the space $\{x\in \bR^S :\sum_v x_v = 0\}$ from the $S$-unit theorem.

The $L^p$ norms for $1\leq p\leq\infty$, which we term the \emph{$L^p$ Weil heights}, naturally make $\cF$ a normed vector space for which we may ask if the space of algebraic $S$-units $U^\alg_S$ is dense in its appropriate codimension $1$ hyperplane, thus satisfying the analogue of part \eqref{enum:first-conclusion} of Dirichlet's theorem. To this end we note that Allcock and Vaaler \cite[Theorems 1-3]{AV} determined the completions of $\cF$ with respect to the $L^p$ norm, which we denote $\cF_p$:
\begin{equation}\label{eqn:completionsFp}
 \cF_p = \begin{cases}
 \{f\in L^1(Y,\lambda) : \int_Y f\,d\lambda =0\} & \text{if }p=1,\\
 L^p(Y,\lambda) & \text{if }1<p<\infty,\\
 C_0(Y) & \text{if }p=\infty.
 \end{cases}
\end{equation}
where $C_0(Y)$ denotes the usual space of continuous functions which vanish at infinity.

\subsection{Main results}
Our main results are the following three theorems, which first provide a generalization of both parts of Dirichlet's result to the larger group of all algebraic $S$-units, and then demonstrate that if the assumption that the archimedean places are included in $S$ is discarded then the space spanned by $\cF(S)$ may in fact be a proper subspace of the hyperplane determined by the product formula.
\begin{thm}\label{thm:1}
 Let $S\subset Y$ be a compact open set containing all of the archimedean places. Then the vector space of algebraic $S$-units
\begin{equation}\label{eqn:VS-space}
 \cF(S) = \{f\in\cF : \supp(f)\subseteq S\}
\end{equation}
 is dense in the closed vector space
\begin{equation}
\cF_p(S,0)=\{f\in\cF_p : \supp(f)\subseteq S\text{ and }\int_S f = 0\}
\end{equation}
 for all $1\leq p\leq \infty$, where $\cF_p$ denotes the completion of $\cF$ under the $L^p$ norm.
\end{thm}
\noindent Theorem \ref{thm:1} generalizes part \eqref{enum:first-conclusion} of Dirichlet's theorem as, by Lemma \ref{lemma:S-iff-compact-open}, 
\begin{equation}
\cF(S) = \phi(U^\alg_T)
\end{equation}
for some number field $K$ and a finite subset $T\subset M_K$ containing all of the archimedean places, where $U^\alg_T$ is defined as in \eqref{eqn:UalgS} and $\phi$ as in \eqref{eqn:phi-function}, so $\cF(S)$ is indeed the vector space (modulo torsion) of algebraic $S$-units. We note that by \eqref{eqn:completionsFp},
\[
 \cF_p(S,0) = \begin{cases}
               \{f\in L^p(Y) : \supp(f)\subseteq S\text{ and }\int_S f = 0\} & \text{ if }1\leq p<\infty,\\
               \{f\in C_0(Y) : \supp(f)\subseteq S\text{ and }\int_S f = 0\} & \text{ if }p=\infty.
              \end{cases}
\]

Regarding the $\bR$-linear independence of $\bQ$-linearly independent vectors, that is, conclusion \eqref{enum:second-conclusion} of Dirichlet's theorem, we prove the following:
\begin{thm}\label{thm:Rdep-Qdep}
 Let $S\subset Y$ be a compact open set containing all of the archimedean places and let $\cF(S)$ be the vector space of $S$-units defined above in \eqref{eqn:VS-space}. Let $f^{(1)},\ldots, f^{(n)}\in \cF(S)$ and suppose there exist elements $r_1,\ldots, r_n\in \bR$, not all zero, such that
 $
  \sum_{i=1}^n r_i f^{(i)} = 0.
 $
 Then there exist $s_1,\ldots,s_n\in\bQ$ such that
 $
  \sum_{i=1}^n s_i f^{(i)} = 0.
 $
\end{thm}

\begin{rmk}\label{rmk:implies-classical}
We note that the classical $S$-unit theorem can be easily obtained assuming the statements of Theorems \ref{thm:1} and \ref{thm:Rdep-Qdep}. To see this, recall from \cite{FM-mahler} that the projection operator $P_K : \cF \ra V_K$ defined in \cite[\S 2.3]{FM-mahler} is a conditional expectation, and thus maps functions supported on $S$ to themselves, and it also acts as the identity on $V_{K,S}$, and thus $P_K(\cF(S))=V_{K,S}$, the $\bQ$-vector space span of the $S$-units of $K$. Since $P_K$ is continuous, Theorem \ref{thm:1} implies that $V_{K,S}$ spans the hyperplane of functions supported on $S$ and being locally constant on the sets $Y(K,v)$, which is part \eqref{enum:first-conclusion} of the classical theorem for the $S$-units of $K$, and since $P_K$ is a linear operator, the conclusion of Theorem \ref{thm:Rdep-Qdep} implies that the elements of of $V_{K,S}$ which are $\bR$-linearly dependent are also $\bQ$-linearly dependent, fulfilling conclusion \eqref{enum:second-conclusion} of the classical $S$-unit theorem and recovering the classical result.
\end{rmk}

As we noted above, our proof uses the classical result, and thus we must include the archimedean places, which consist precisely of the set $Y(\bQ,\infty)$, in our set $S$. In fact, it is not possible to entirely remove this assumption, and Theorem \ref{thm:1} is sharp in the following precise sense. Say $S$ is \emph{defined} over a field $K$ if we can write $S$ as a union of sets $Y(K,v)$ where $v$ ranges \emph{only} over places of $K$. Thus for example if $v$ is an infinite place of a real quadratic number field $K$ then $S=Y(K,v)$ is defined over $K$ and any extension thereof but not defined over $\bQ$ since $S\subsetneqq Y(\bQ,\infty)$. Then we have:
\begin{thm}\label{thm:2}
 Let $S\subset Y$ be a compact open set such that $S\cap Y(\bQ,\infty)=\varnothing$. Let $S$ be defined over a number field $K$ which has at least one real archimedean place and suppose at least two places $v\neq w\in M_K$ are such that $Y(K,v)\cup Y(K,w)\subseteq S$. Then the closure of the $S$-units space $\cF(S)$ in the $L^p$ norm is a proper subspace of $\cF_p(S,0)$.
\end{thm}

\begin{example}
As an example, let $S=Y(\bQ,2)\cup Y(\bQ,3)$. Then the vector space $\cF(S)$ consists of nonzero algebraic numbers which along with all of their conjugates in the complex plane lie on the unit circle and possess nontrivial valuations only on places lying over $2$ and $3$. We note that $S$ is in fact defined over $\bQ$ and thus Theorem \ref{thm:2} applies, so the space $\cF(S)$ is not dense in $\cF_p(S,0)$. In fact, if we restrict to elements of $\cF(S)$ which arise from $\bQ$, then this result is clear: there is no element of $\bQ^\times$ which can have a positive valuation over $2$ and a negative valuation over $3$ and vanish on the infinite place by the $\bQ$-linear independence of $\log 2$ and $\log 3$, while in the completion $\cF_p(S,0)$ we can easily construct such a function as we have no such number theoretic obstructions. The idea of the proof of Theorem \ref{thm:2} is that by using projections, we can demonstrate that it is impossible that some element supported on $S$ from an extension of $\bQ$ can make up for these missing functions in attempting to approximate elements of $\cF_p(S,0)$, and thus the closure of $\cF(S)$ is a proper subspace.
\end{example}

Lastly, we demonstrate by the next example that even in the case where the vector space of $S$-units spans a proper subspace of the hyperplane $\cF_p(S,0)$, the space need not be trivial.
\begin{example}
Let $S=Y(\bQ,5)\cup Y(\bQ,7)$. By Theorem \ref{thm:2}, the space $\cF(S)$ is not dense in $\cF_p(S,0)$ for any $1\leq p\leq\infty$. Let $K=\bQ(i)$ where $i^2=-1$, and let $v$ be the nonarchimedean place of $K$ associated to the prime ideal $(2-i)$ and $w$ the place associated to $(2+i)$. Notice that since $5=(2-i)(2+i)$, we have $Y(\bQ,5)=Y(K,v)\cup Y(K,w)\subseteq S$. Let $\al=(2-i)/(2+i)\in K^\times$, and observe that $\fal$ is supported only on $Y(\bQ,5)\subseteq S$ and therefore $\fal\in \cF(S)$. Thus, $\cF(S)\neq \{0\}$ is nontrivial, as claimed, although it is not dense in $\cF_p(S,0)$.
\end{example}

\section{Proofs}\label{section:2}
\subsection{Compact open sets of places and Proof of Lemma \ref{lemma:S-iff-compact-open}}\label{section:2.1}
In this section we prove Lemma \ref{lemma:S-iff-compact-open}, as well as some other results which are not essential to the proofs of our main theorems but are of interest in themselves.
\begin{proof}[Proof of Lemma \ref{lemma:S-iff-compact-open}]
Our proof follows the same approach as in Lemma 1 of \cite{AV}. Since $S$ is open in $Y$, we may choose for each $y\in S$ a number field $L^{(y)}$ such that $y$ lies over the place $v^{(y)}$ of $L^{(y)}$ and $Y(L^{(y)},v^{(y)})\subset S$. Notice that the collection $\{Y(L^{(y)},v^{(y)}): y\in S\}$ is an open cover of the compact set $S$ and as such it possesses a finite subcover,
\[
 \{ Y(L_n,v_n) \}_{n=1}^N
\]
 where the $L_n$ are number fields and $v_n$ is a place of $L_n$. Let $L=L_1\cdots L_N$ be the compositum. Then clearly each set $Y(L_n,v_n)$ is a finite union of $Y(L,w)$ where $w$ ranges over the places of $L$ which lie over $v_n$. Thus, we can select finitely many places $w_i$ of $L$ such that $S=\bigcup_i Y(L,w_i)$ and we have the desired result.
\end{proof}
Although Lemma \ref{lemma:S-iff-compact-open} suffices for our purposes, with only a little more work one can see that there exists a unique minimal field of definition of any compact open $S$. We will now demonstrate this result. Recall that $G$ acts on the set $Y$ in a well-defined manner given by $\|\al\|_{\sigma{y}} =\|\sigma^{-1}\al\|_y$ (see \cite{AV} for more details). We make the following definition:
\begin{defn}
 Let $S\subset Y$ be a compact open set, and let $K$ be a number field. We say $S$ is \emph{defined over $K$} if there exists a set $T\subset M_K$ (necessarily finite) such that
 \[
  S = \bigcup_{v\in T} Y(K,v).
 \]
\end{defn}
\noindent Notice that by Lemma \ref{lemma:S-iff-compact-open} there always exists \emph{some} number field $K$ such that $S$ is defined over $K$.
\begin{lemma}\label{lemma:S-over-K-characterization}
 A compact open set of places $S$ is defined over a number field $K$ if and only if for any place $v\in M_K$ such that $S\cap Y(K,v)\neq \varnothing$ we have $Y(K,v)\subseteq S$.
\end{lemma}
\begin{proof}
 If $S$ is defined over $K$, this is obvious from the definition and the fact that the sets $Y(K,v)$ are disjoint from each other for different places $v$. Suppose now that the conclusion holds. Define
 \[
  T=\{v\in M_K : S\cap Y(K,v)\neq \varnothing \}.
 \]
 Since $Y=\bigcup_{v\in M_K} Y(K,v)$, the collection $\{Y(K,v) : v\in T\}$ forms an open cover of $S$. By compactness, $T$ must therefore be finite, and $S$ must in fact be the union of the sets $Y(K,v)$ for $v\in T$.
\end{proof}
\begin{prop}
 Let $S$ be a compact open set. Then $S$ is defined over $K$ if and only if $K$ is a finite extension of the field 
 \begin{equation}
 k=\Fix(\Stab_G(S))
 \end{equation}
 where $\Stab_G(S)=\{\sigma\in G : \sigma(S)=S\}$ and $\Fix(\cdot)$ denotes the usual fixed field of a subgroup of the absolute Galois group.
\end{prop}
Thus $k$ is the unique minimal field over which $S$ is defined, and such a field may be naturally associated to any compact open set $S$ as the \emph{field of definition} of $S$.
\begin{proof}
 First, let us show that any field over which $S$ is defined is an extension of $k$ (and in particular, that $k$ is indeed a number field, so the subgroup $\Stab_G(S)$ must have finite index in $G$). By Lemma \ref{lemma:S-iff-compact-open}, $S$ is defined over some number field $K$. Clearly $\Gal(\Qbar/K)\subseteq \Stab_G(S)$, since $S$ is a union of sets of the form $Y(K,v)$ but $\Gal(\Qbar/K)$ must fix such sets if it is to act trivially on $K$.  (That is, if $y\in Y(K,v)$ but $\sigma y$ were not, then for some $\alpha\in K$ we would have $\|\alpha\|_y\neq\|\alpha\|_{\sigma y}$ as $y$ and $\sigma y$ must lie over distinct places of $K$, and hence $\|\sigma^{-1}\alpha\|_v=\|\sigma^{-1}\alpha\|_y=\|\alpha\|_{\sigma y}\neq\|\alpha\|_v$, which implies that $\alpha\neq\sigma\alpha$, and so $\sigma\notin\Gal(\Qbar/K)$.) But then $k\subseteq K$ by the Galois correspondence, and in particular $k$ is a number field.

 Now let $k=\Fix(\Stab_G(S))$. We will show that $S$ is in fact defined over $k$. Assume it is not; then by Lemma \ref{lemma:S-over-K-characterization}, there must exist a place $v$ of $k$ such that $S\cap Y(k,v)\neq\varnothing$, but $Y(k,v)\not\subseteq S$. Since the sets of the form $Y(L,w)$ for $L$ a finite Galois extension form a basis for the topology (see \cite{AV}), there exists some such $L$ which contains $k$ such that $Y(L,w)\subsetneqq S\cap Y(k,v)$. Then the place $v$ of $k$ splits into places of $L$, including $w$. Let $w'$ be a place of $L$ such that $Y(L,w')\subset Y(k,v)\setminus S$. By the same reasoning such an $L$ exists, after replacing $L$ with a larger Galois extension if necessary. Then there exists $\sigma\in \Gal(L/k)$ such that $\sigma w' = w$. Lifting $\sigma$ to $G$, we see that $\sigma(Y(L,w'))=Y(L,w)$ and $\sigma\in \Gal(\Qbar/k)\subset G$. But then $\sigma\in \Stab_G(S)$ by definition of $k$ as the fixed field of $\Stab_G(S)$, but for any $y\in Y(L,w)$, $\sigma(y)\not\in S$ because $\sigma(y)\in Y(L,w')\subset Y(K,v)\setminus S$ by construction. But this is a contradiction to $\sigma$ being in the stabilizer of the set $S$.
\end{proof}

\subsection{Proof of Theorem \ref{thm:1}}
The proof will proceed in a series of lemmas. For an open set $E\subset Y$, we let $LC(E)$ denote the set of all locally constant functions on $E$, that is, functions $f$ such that $\supp(f)\subseteq E$ and at each $y\in E$, there exists an open neighborhood $N$ of $y$ such that $f$ is constant on $N$. We let $LC_c(E)$ denote the locally constant functions on $E$ with compact support. We assume $1\leq p\leq \infty$ and we denote by $\|\cdot\|_p$ the usual $L^p$ norm throughout.

\begin{lemma}\label{lemma:3}
 Let $f\in LC_c(Y)$. Then there exists a number field $K$ such that
\[
 f(y) = \sum_{v\in M_K} x_v\cdot\chi_{Y(K,v)}(y)
\]
 where $x_v\in\bR$ and is zero for almost all $v$, and $\chi_E$ denotes the characteristic function of the set $E$.
\end{lemma}
\begin{proof}
This is \cite[Lemma 4]{AV}.
\end{proof}
\begin{lemma}\label{lemma:2}
 For any given $\ep>0$ and $f\in C(Y(K,v))$ there exists a $g\in LC(Y(K,v))$ such that $\|f-g\|_p<\ep$, where $K$ is a number field and $v\in M_K$ and $1\leq p\leq \infty$.
\end{lemma}
\begin{proof}
 This result is proven in \cite{AV} in the course of proving Theorems 1-3.
\end{proof}
\begin{lemma}\label{lemma:LCS-dense-in-CS}
 Let $S\subset Y$ be a compact open set. For any given $\ep>0$ and $f\in C(S)$ there exists a $g\in LC(S)$ such that $\|f-g\|_p<\ep$.
\end{lemma}
\begin{proof}
 By Lemma \ref{lemma:S-iff-compact-open} there exists a field $K$ and a finite set of places $T\subset M_K$ such that $S=\bigcup_{v\in T} Y(K,v)$. Observe that when each space is endowed with the $L^p$ norm we have:
\[
 C(S) = {\bigoplus_{v\in T}}^{(p)} C(Y(K,v))
\]
 where the direct sum notation implies that each $C(Y(K,v))$ space is endowed with the $L^p$ norm and the finite number of vector spaces are combined with the $\ell^p(T)$ norm. The result then follows immediately from Lemma \ref{lemma:2} above.
\end{proof}

$S$ will denote throughout the remainder of this section a compact open set of $Y$. Let
\[
 LC(S,0) = \{f\in LC(S) : \int_S f =0\},\quad C(S,0) = \{f\in C(S) : \int_S f =0\},
\]
and
\[
 \cF_p(S,0)=\{f\in\cF_p : \supp(f)\subseteq S\text{ and }\int_S f = 0\}
\]
(We note in passing that $\cF_1(S,0)=\cF_1(S)=\{f\in\cF_1 : \supp(f)\subseteq S\}$, but that this is not the case for $p>1$.) 

\begin{lemma}
 For any given $\ep>0$ and $f\in LC(S,0)$ there exists $g\in \cF(S)$ such that $\|f-g\|_p<\ep$, where $\cF(S)$ denotes the $S$-unit space defined in \eqref{eqn:VS-space}. 
\end{lemma}
\begin{proof}
 Suppose we have $f\in LC(S,0)$. By Lemma \ref{lemma:3}, $f$  has the form
\begin{equation}\label{eqn:1}
 f(y) = \sum_{v\in T} x_v \cdot \chi_{Y(L,v)}(y)
\end{equation}
 where $T=\{v\in M_L : Y(L,v)\subseteq S\}$, $L$ is some number field over which $S$ is defined, and $x_v\in\bR$. Notice that the $\bQ$-vector space span of the $S$-units of $L$ in $\cF$ is precisely $\cF(S)\cap V_L$. Let
\[
 \bR(T,0) = \{ x\in\bR^{T} : \sum_{v\in T} \frac{[L_v:\bQ_v]}{[L:\bQ]} x_v = 0\}.
\]
 Observe that we have natural embeddings
 \begin{align*}
  \varphi: V_L\cap \cF(S) &\hookrightarrow \bR(T,0)\\
    \fal  &\mapsto (\log\|\al\|_v)_{v\in T}
 \end{align*}
 and 
 \begin{align*}
  \theta: \bR(T,0) &\hookrightarrow \cF_p(S,0)\\
    (x_v)_{v\in T}  &\mapsto \sum_{v\in T} x_v \cdot \chi_{Y(L,v)}(y)
 \end{align*}
 and that $\theta\circ\varphi:V_L\cap \cF(S)\hookrightarrow \cF_p(S,0)$ is the inclusion map. (Notice that the embedding $\varphi$ is a rescaled version of the logarithmic embedding \eqref{eqn:logarithmic-embedding} used in the statement of Dirichlet's theorem.) Observe that if $x_v$ is as in \eqref{eqn:1} above, the element $x=(x_v)\in \bR(T,0)$ has $\theta(x)=f$. By Dirichlet's $S$-unit theorem for the $S$-units of $L$, we see that the image $\varphi(\cF(S)\cap V_L)$ is a $\bQ$-vector space of full rank in $\bR(T,0)$ and is thus dense (in any norm, as all vector space norms are equivalent on finite dimensional spaces). Observe that the $L^p$ norm on $\cF_p(S,0)$ induces a norm on $\bR(T,0)$ which agrees with its restriction to $V_L$. Thus, since $\bR(T,0)$ is finite dimensional, we see that for any $\ep>0$ we can approximate $x\in \bR(T,0)$ by an element of $g\in \cF(S)\cap V_L$ with
\[
 \|g-f\|_p= \|g-\theta(x)\|_p= \|\varphi(g)-x\| <\ep,
\]
 and the proof is complete.
\end{proof}

Define the map
\begin{equation}\label{eqn:PS-defn}
 P_S : \cF_p \ra \cF_p(S,0)
\end{equation}
via
\begin{equation}
 (P_Sf)(y)=\chi_S(y)\bigg(f(y) - \frac{1}{\lambda(S)}\int_S f(z)\,d\lambda(z)\bigg).
\end{equation}
\begin{lemma}
 $P_S$ is a continuous projection in $L^p$ norm onto its range $\cF_p(S,0)$ for $1\leq p\leq\infty$.
\end{lemma}
\begin{proof}
First, observe that $P_S$ is well-defined, as the integral of an $L^p$ function on a compact set is well-defined, and note that
\begin{align*}
 \|P_S f\|_p &\leq \|\chi_S(y)\cdot f(y)\|_p + \left|\int_S f(z)\,d\lambda(z)\right| \\
 &\leq \|f\|_p + \lambda(S)^{1-1/p} \|f\cdot \chi_S \|_p\leq (1+\lambda(S)^{1-1/p})\|f\|_p
\end{align*}
and thus $P_S$ is continuous with $L^p$ operator norm\footnote{Recall that the \emph{(operator) norm $\|T\|$} of a linear transformation $T$ on a normed vector space $V$ is defined as the infimum of the values $C>0$ such that $\|Tx\|\leq C\|x\|$ for all $0\neq x\in V$ (or $+\infty$ if no such value exists). $T$ is continuous if and only if it has a finite operator norm.}
\[
 \|P_S\|_p\leq 1+\lambda(S)^{1-1/p}.
\]
It is easy to see that ${P_S}^2=P_S$, and thus $P_S$ is a projection onto the subspace $\cF_p(S,0)$, on which it obviously acts as the identity.
\end{proof}

\begin{lemma}
 For any given $\ep>0$ and $f\in C(S,0)$ there exists a $g\in LC(S,0)$ such that $\|f-g\|_p<\ep$. 
\end{lemma}
\begin{proof}
 Let $f\in C(S,0)\subset C(S)$. By Lemma \ref{lemma:LCS-dense-in-CS} we can choose $h\in LC(S)$ such that $\|f-h\|_p<\ep$. Observe that $P_S f= f$, so
\[
 \|f-P_Sh\|_p = \|P_S(f-h)\|_p \leq \|P_S\|_p\cdot \|f-h\|_p < \|P_S\|_p\cdot\ep.
\]
 Observe that $P_Sh$ differs from $h$ by a constant function on $S$ (namely, $-\frac{1}{\lambda(S)}\int_S h$) and therefore is also an element of $LC(S)$, and in fact, lies in $LC(S,0)$. Thus we may take $g=P_Sh$ and the proof is complete.
\end{proof}
\begin{lemma}
 For any given $\ep>0$ and $f\in \cF_p(S,0)$ there exists a $g\in C(S,0)$ such that $\|f-g\|_p<\ep$.
\end{lemma}
\begin{proof}
 Let $f\in\cF_p(S,0)$. Since $C_c(Y)$ is dense in $\cF_p$, we can choose $h\in C_c(Y)$ such that $\|f-h\|_p<\ep$. Then observe that $P_Sf=f$, so
\[
 \|f-P_Sh\|_p = \|P_S(f-h)\|_p \leq \|P_S\|_p\cdot \|f-h\|_p < \|P_S\|_p\cdot\ep,
\]
 and $P_S h\in C(S,0)$ as well since it is a continuous function restricted to a compact open set minus a function constant on that compact open set. Thus we may take $g=P_S h$ and the proof is complete.
\end{proof}

\noindent Combining the above lemmas, we see that given any $f\in \cF_p(S,0)$ there exists a $g\in \cF(S)$ such that $\|f-g\|_p<3\ep$ by the triangle inequality. Since $\cF_p(S,0)$ is closed (it is defined as the null space of the linear functional $\int_S$ on $\cF_p$) the proof of Theorem \ref{thm:1} is complete.

\subsection{Proof of Theorem \ref{thm:Rdep-Qdep}}
Suppose the equivalence class of $f^{(i)}$ has a representative $\al_i\in\Qbar^\times$. Let $K=\bQ(\al_1,\ldots,\al_n)$. Then $f^{(i)}\in V_K$ for each $i$. Let $T=\{v\in M_K : Y(K,v)\subseteq S\}$. Each $f^{(i)}$ corresponds exactly to a vector $x_i=(x_{i,v})_{v\in T}\in\bR^T$ under the logarithmic embedding defined in \eqref{eqn:logarithmic-embedding} above, specifically, we let
\[
 x_{i,v} = \frac{[K_v:\bQ_v]}{[K:\bQ]} f^{(i)}(y)\quad\text{where}\quad y\in Y(K,v).
\]
(Recall that since $f^{(i)}\in V_K$, it is constant on the sets $Y(K,v)$, so this is well-defined.) Notice that
\[
 \bigoplus_{i=1}^n \bZ\cdot x_i\subseteq l(U_{K,T})
\]
where $U_{K,T}$ denotes the usual $S$-unit subgroup of the field $K$. By the classical Dirichlet $S$-unit theorem, any $\bR$-linear dependence amongst the vectors $x_i$ implies a $\bQ$-linear dependence by the discreteness of the lattice $l(U_{K,T})$ (as discussed in part \eqref{enum:second-conclusion} of Dirichlet's $S$-unit Theorem above). But then any linear dependence amongst the $x_i$ vectors is obviously equivalent to the same dependence amongst the $f^{(i)}$ functions by construction, and the proof is complete. \hfill $\qed$

\subsection{Proof of Theorem \ref{thm:2}}
We begin with the following lemma:
\begin{lemma}\label{lemma:thm2-main-lemma}
Suppose $S$ is a compact open subset of $Y$ which is defined over a number field $K$ which has at least one real archimedean place. If $S\cap Y(\bQ,\infty)=\varnothing$, then the subspace
\[
 V_{K,S} = \{f\in V_K : \supp(f)\subseteq S\} \subset \cF(S)
\]
is trivial, that is, $V_{K,S}=\{0\}$. 
\end{lemma}
\begin{proof}
By assumption there is a finite set of places $T\subset M_K$ such that
\[
 S = \bigcup_{v\in T} Y(K,v).
\]
Since $S\cap Y(\bQ,\infty)=\varnothing$, clearly $v\nmid\infty$ for $v\in T$. For each $v$ we have an associated prime ideal $\cP_v\subset\cO_K$, where $\cO_K$ denotes the ring of integers of $K$. By raising to the class number $h$, we can make $\cP_v$ principal so that $\cP_v^h =(\al_v)$. Choose such an $\al_v$ for each $v\in T$ and let $\xi_1,\ldots, \xi_n$ be a $\bZ$-module basis for the unit group modulo torsion of $K$ (such a basis exists by Dirichlet's theorem). Notice that each $\al\in K^\times$ which is supported only on $T$ then has an expansion
\[
 (\al^h) = \prod_{v\in T} {\cP_v}^{h\cdot \ord_v(\al)}=(\al_v)^{\ord_v(\al)}
\]
by the unique factorization of prime ideals, where $\ord_v :K^\times\ra \bZ$ is the usual $v$-adic discrete valuation. Since this factorization holds modulo the units of $K$, it now follows that any $0\neq f\in V_{K,S}$ must have a unique expansion
\[
 f = \sum_{v\in T} r_v f_{\al_v} + \sum_{i=1}^n s_i f_{\xi_i},\quad\text{where}\quad r_v,s_i\in\bQ.
\]
Since $\supp(f)\subseteq S$, we have $f(y)=0$ for all $y|\infty$. Let $y\in Y(K,w)$ where $w$ is a real archimedean place of $K$. Then
\begin{align*}
 f(y) = 0 &= \sum_{v\in T} r_v f_{\al_v}(y) + \sum_{i=1}^n s_i f_{\xi_i}(y)\\
 &= \sum_{v\in T} r_v \log\|\al_v\|_y + \sum_{i=1}^n s_i \log\|\xi_i\|_y.
\end{align*}
From this it follows that
\[
 \big\|\prod_{v\in T} \al_v^{r_v}\big\|_y = \big\|\prod_{i=1}^n \xi_i^{-s_i}\big\|_y.
\]
Raising each side to an appropriate integral power, we find there exist integers $m_v,m_i$ such that:
\[
 \big\|\prod_{v\in T} \al_v^{m_v}\big\|_y = \big\|\prod_{i=1}^n \xi_i^{-m_i}\big\|_y.
\]
Observe that, since $\|\cdot\|_y$ is a real absolute absolute value, we have for the associated real embedding:
\[
 \prod_{v\in T} \al_v^{m_v} = \pm \prod_{i=1}^n \xi_i^{-m_i}
\]
But this is impossible, as the left hand side is an algebraic number with nontrivial nonarchimedean valuations, while the right hand side is an algebraic unit. Therefore, no nonzero element $f$ of $V_{K,S}$ can exist.
\end{proof}

\begin{proof}[Proof of Theorem \ref{thm:2}]
 Suppose that $\cF(S)$ is in fact dense in $\cF_p(S,0)$. By assumption we have a number field $K$ over which $S$ is defined and which has at least one real archimedean place. Let $LC(S,0,K)\subset \cF_p(S,0)$ denote the subspace of continuous functions supported on $S$ with integral zero which are locally constant on the sets $Y(K,v)$. Since by assumption there exist at least two places $v_1,v_2$ of $K$ such that $Y(K,v_i)\subseteq S$, the space $LC(S,0,K)$ is nonzero (consider the function which takes the value $+1$ on $Y(K,v_1)$, $-1$ on $Y(K,v_2)$, and vanishes elsewhere). Let $0\neq f\in LC(S,0,K)$, and normalize $f$ so that $\|f\|_p=1$. If $\cF(S)$ is indeed dense, then for any $\ep>0$ there exists some $g\in \cF(S)$ such that $\|f-g\|_p < \ep$. Let $P_K:\cF\ra V_K$ denote the orthogonal projection as defined in \cite[\S 2.3]{FM-mahler}. We recall from \cite{FM-mahler} that $P_K$ is a norm one projection with respect to all $L^p$ norms, $1\leq p\leq\infty$, and thus extends by continuity to the completion $\cF_p$. We therefore have
\begin{equation}\label{eqn:approx-byPKg}
 \|f-P_K g\|_p =\|P_K(f-g)\|_p \leq \|f-g\|_p <\ep,
\end{equation}
 since $P_Kf=f$ by the fact that $P_K$ is a conditional expectation on the sets $Y(K,v)$ (see \cite[Lemma 2.10]{FM-mahler}). But $P_K g\in \cF(S)\cap V_K = V_{K,S}=\{0\}$ by Lemma \ref{lemma:thm2-main-lemma}, so $P_Kg=0$ and 
 $
  \|f\|_p<\ep 
 $
 by \eqref{eqn:approx-byPKg} above, but this is a contradiction to our normalization of $f$ when we take any  $0<\ep<1$.
\end{proof}


\begin{thebibliography}{AAA}

\bibitem{AV}
D. Allcock, J.D. Vaaler. \emph{A Banach Space determined by the Weil Height.} Acta Arith. 136 (2009), no. 3, 279--298.

\bibitem{FM-extremal}P. Fili, Z. Miner. \emph{Norms extremal with respect to the Mahler measure.} Preprint available at {\tt http://arxiv.org/abs/1006.5503}.

\bibitem{FM-mahler}P. Fili, Z. Miner. \emph{Orthogonal decomposition of the space of algebraic numbers and Lehmer's problem.} Preprint available at {\tt http://arxiv.org/abs/0911.1975v1}.

\bibitem{Nark}W. Narkiewicz. Elementary and analytic theory of algebraic numbers. Second edition. Springer-Verlag, Berlin; PWN---Polish Scientific Publishers, Warsaw, 1990. xiv+746 pp.

\end{thebibliography}
\end{document}